\documentclass[11 pt]{amsart}
\usepackage{amscd,amsfonts,amssymb,amsmath}
\usepackage{hyperref}
\usepackage{epsfig}
\usepackage{wasysym}
\usepackage{latexsym}
\usepackage{wasysym}
\usepackage{amssymb}
\usepackage{enumerate}
\usepackage{amsthm}
\usepackage{hyperref}
\usepackage{amscd}
\usepackage{latexsym}
\usepackage{amssymb}
\usepackage{amscd}
\usepackage[all,cmtip]{xy}
\newcommand{\Core}{\operatorname{Core}}

\newtheorem{theorem}{Theorem}[section]
\newtheorem{corollary}[theorem]{Corollary}
\newtheorem{lemma}[theorem]{Lemma}
\newtheorem{proposition}[theorem]{Proposition}
\theoremstyle{definition}

\newtheorem{question}[theorem]{Question}

\newtheorem{remark}[theorem]{Remark}
\numberwithin{equation}{subsection}

\newcommand{\Alex}{\operatorname{Alex}}
\newcommand{\Aut}{\operatorname{Aut}}

\newcommand{\QInn}{\operatorname{QInn}}
\newcommand{\Conj}{\operatorname{Conj}}

\newcommand{\Inn}{\operatorname{Inn}}

\newcommand{\Ha}{\operatorname{H}}

\newcommand{\R}{\operatorname{R}}

\newcommand{\Z}{\operatorname{Z}}

\newcommand{\id}{\mathrm{id}}

\begin{document}
\title{Automorphism groups of quandles and related groups}
\author{Valeriy G. Bardakov}
\author{Timur R. Nasybullov}
\author{Mahender Singh}

\address{Sobolev Institute of Mathematics, 4 Acad. Koptyug avenue, 630090, Novosibirsk, Russia, and Novosibirsk State University, 2 Pirogova Str., 630090, Novosibirsk, Russia.}
\email{bardakov@math.nsc.ru}

\address{Katholieke Universiteit Leuven KULAK, 53  E.~Sabbelaan, 8500, Kortrijk, Belgium.}
\email{timur.nasybullov@kuleuven.be}

\address{Department of mathematical sciences, Indian Institute of Science Education and Research (IISER) Mohali, Sector 81,  S. A. S. Nagar, P. O. Manauli, 140306, Punjab, India.}
\email{mahender@iisermohali.ac.in}


\begin{abstract}
In this paper we  study  different questions concerning automorphisms of quandles. For a conjugation quandle $Q={\rm Conj}(G)$ of a group $G$ we determine several subgroups of ${\rm Aut}(Q)$ and find necessary and sufficient conditions when these subgroups coincide with the whole group ${\rm Aut}(Q)$. In particular, we prove that  ${\rm Aut}({\rm Conj}(G))=\Z(G)\rtimes {\rm Aut}(G)$ if and only if either $\Z(G)=1$ or $G$ is one of the groups $\mathbb{Z}_2$, $\mathbb{Z}_2^2$  or $\mathbb{Z}_3$, what solves \cite[Problem 4.8]{BDS}. For a big list of Takasaki quandles $T(G)$ of an abelian group $G$  with $2$-torsion we prove that the group of inner automorphisms ${\rm Inn}(T(G))$ is a Coxeter group, what extends the result \cite[Theorem 4.2]{BDS} which describes $\Inn(T(G))$ and ${\rm Aut}(T(G))$ for an abelian group $G$ without $2$-torsion. We study automorphisms of certain extensions of quandles and determine some interesting subgroups of the automorphism groups of these quandles.
Also we classify finite quandles $Q$ with $3\leq k$-transitive action of ${\rm Aut}(Q)$.
\end{abstract}
\maketitle
\section{Introduction}
 A quandle is an algebraic system whose axioms are derived from the Reidemeister moves on oriented link diagrams. Such algebraic system were firstly introduced by Joyce \cite{Joyce} (under the name  ``quandle'') and Matveev \cite{Matveev} (under the name ``distributive groupoid'') as an invariant for knots in $\mathbb{R}^3$. More precisely,  to each oriented diagram $D_K$ of an oriented knot $K$ in $\mathbb{R}^3$ one can associate the quandle $Q(K)$ which does not change if we apply the Reidemeister moves to the diagram $D_K$.  Moreover, Joyce and Matveev proved that two knot quandles $Q(K_1)$ and $Q(K_2)$ are isomorphic if and only if $K_1$ and $K_2$ are weak equivalent, i.~e. there exists a homeomorphism of $\mathbb{R}^3$ (probably, orientation reversing) which maps $K_1$ to $K_2$.  Over the years, quandles have been investigated by various authors in order to construct new
invariants for knots and links (see, for example, \cite{Carter, Kamada, Nelson}).

The knot quandle is a very strong invariant for knots in $\mathbb{R}^3$, however, usually it is very difficult to understand if two knot quandles are isomorphic or not. Sometimes homomorphisms from knot quandles to some simpler quandles provide useful information which helps to understand if two knot quandles are not isomorphic. For example, Fox's $n$-colourings (which are useful invariants for links in $\mathbb{R}^3$) are representations of the knot quandle into the dihedral quandle $\R_n$ on $n$ elements \cite{Fox,prz}. This leads to the necessity of studying quandles which are not necessarily knot quandles from the algebraic point of view. Quandles turned out to be useful in different branches of algebra, topology and geometry since they have connections to several different topics such
as permutation groups \cite{HulStaVoj}, quasigroups \cite{Sta}, symmetric spaces \cite{Tam}, Hopf
algebras \cite{Andruskiewitsch}.

In this paper we investigate different questions concerning automorphisms of quandles. We also study connections between groups of automorphisms of quandles and groups of automorphisms of some groups closely related to quandles. 
 At this point, our approach is purely algebraic and we make no reference to knots. However, it would be interesting to explore implications of these results in knot
theory.

The paper is organized as follows. In Section \ref{sec2} we recall definitions and notions used in the paper and give several examples of quandles. In Section \ref{sec3} we recall the notion of the enveloping group $G_Q$ of a quandle $Q$ and prove some properties of this group. In particular, we study relations between automorphism groups of $Q$ and $G_Q$ (Proposition \ref{inclus}). In Section \ref{sec4} we study the automorphism group of the conjugation quandle $Q={\rm Conj}(G)$ of a group $G$ and find some groups which are imbedded into ${\rm Aut}(Q)$. In particular, we prove that ${\rm Aut}({\rm Conj}(G))=\Z(G)\rtimes{\rm Aut}(G)$ if and only if either $\Z(G)=1$ or $G$ is one of the groups $\mathbb{Z}_2$, $\mathbb{Z}_2^2$ or $\mathbb{Z}_3$ (Theorem \ref{solution}). In Section \ref{sec5} we study automorphism groups of core quandles (in particular, of Takasaki quandles). In Section \ref{sec6} we classify quandles $Q$ such that the automorphism group ${\rm Aut}(Q)$ acts $3\leq k$-transitively on $Q$ (Theorem \ref{tran}). In Section \ref{sec7} we study automorphisms of certain extensions of quandles and determine some interesting subgroups of the automorphism groups of these quandles
(Theorem \ref{exterrr}). In Section \ref{sec8} we compare the group of inner automorphisms and the group of quasi-inner automorphisms for quandles and give a simple solution of the analogue of the Burnside's question for quandles (Proposition \ref{dihburn}). Finally, in Section \ref{sec9} we introduce some new ways of constructing quandles from groups and quandles (Propositions \ref{new1} and \ref{new2}).
\subsection*{Acknowledgement} The authors are greatful to Jente Bosmans from Vrije Universiteit Brussel for several remarks and suggestions on Section \ref{sec9}. The results given in Sections \ref{sec3}, \ref{sec5}, \ref{sec7},  \ref{sec8}, \ref{sec9} are supported by the Russian Science Foundation Project 16-41-02006 (V.~Bardakov) and the DST-RSF Project INT/RUS/RSF/P-2 (M.~Singh). The results presented in Sections \ref{sec4}, \ref{sec6} are supported by the Research Foundation -- Flanders (FWO), grant  12G0317N (T.~Nasybullov). We are strongly grateful to them.
\bigskip
\section{Definitions, notations and examples}\label{sec2}
\textit{A quandle $Q$} is an algebraic system with one binary operation $(x,y)\mapsto x*y$ which satisfies the following three axioms:
\begin{enumerate}
\item $x*x=x$ for all $x\in Q$,
\item the map $S_x:y\mapsto y*x$ is a bijection on $Q$ for all $x\in Q$,
\item $(x*{y})*z=(x*z)*({y*z})$ for all $x,y,z\in Q$.
\end{enumerate}
The simplest example of a quandle is \textit{the trivial quandle on a set $X$}, that is the quandle $Q=(X,*)$, where $x*y=x$ for all $x\in X$. A lot of interesting examples of quandles come from groups. 

Let $G$ be a group. For elements $x,y\in G$ denote by $x^y=y^{-1}xy$ \textit{the conjugate of $x$ by $y$}. For an arbitrary integer $n$ the set $G$ with the operation $x*y=x^{y^n}=y^{-n}xy^n$ forms a quandle. This quandle is called \textit{the $n$-th conjugation quandle of the group $G$} and is denoted by ${\rm Conj}_n(G)$. For the sake of simplicity we will use the symbol ${\rm Conj}(G)$ to denote the $1$-st conjugation quandle ${\rm Conj}_1(G)$. Note that a quandle ${\rm Conj}_0(G)$ is trivial. 

If we define another operation $*$ on the group $G$, namely $x*y=yx^{-1}y$, then the set $G$ with this operation also forms a quandle. This quandle is called \textit{the core quandle of a group $G$} and is denoted by $\Core(G)$. In particular, if $G$ is an abelian group, then the quandle $\Core(G)$ is called \textit{the Takasaki quandle of the abelian group $G$} and is denoted by $T(G)$. Such quandles were studied by Takasaki in \cite{Takasaki}. If $G=\mathbb{Z}_n$ is the cylcic group of order $n$, then the Takasaki quandle $T(\mathbb{Z}_n)$ is called \textit{the dihedral quandle on $n$ elements} and is denoted by $\R_n$. In the quandle $\R_{n}=\{a_0, a_1, \ldots, a_{n-1}\}$ the operation is given by the rule $a_i * a_j = a_{2j-i({\rm mod}~n)}$.

If $\varphi$ is an automorphism of an abelian group $G$, then the set $G$ with the operation $x*y=\varphi(xy^{-1})y$ forms a quandle. This quandle is called \textit{the Alexander quandle of the abelian group $G$ with
respect to the automorphism $\varphi$} and is denote by $\Alex(G,\varphi)$. The Takasaki quandle $T(G)$ is a particular case of the Alexander quandle $\Alex(G,\varphi)$ for $\varphi:x\mapsto x^{-1}$. Alexander quandles were studied, for example, in \cite{BDS, Clark1, Clark}.

The subset $\Z(Q)=\{x\in Q~|~x*y=x~\text{for all}~y\in Q\}$ of a quandle $Q$ is called \textit{the center of the quandle $Q$}. In particular, if $Q={\rm Conj}(G)$, then the center $\Z(Q)$ of the quandle $Q$ coincides with the center $\Z(G)$ of the group $G$. For $n\neq\pm1$ the sets $\Z({\rm Conj}_n(G))$ and $\Z(G)$ do not have to coincide.

A bijection $\varphi:Q\to Q$ is called \textit{an automorphism of  a quandle $Q$} if for every elements $x,y\in Q$ the equality $\varphi(x*y)=\varphi(x)*\varphi(y)$ holds. In particular, if $Q$ is the trivial quandle, then the group ${\rm Aut}(Q)$ of automorphisms of the quandle $Q$ consists of all permutations of the elements from $Q$, i.~e. ${\rm Aut}(Q)=\Sigma_{|Q|}$ is  the full permutation group of $|Q|$ elements	.
 From the second and the third axioms of a quandle it follows that the map $S_x:y\mapsto y*x$ is an automorphism of a quandle. The subgroup $\langle S_x~|~x\in Q\rangle$ of the group ${\rm Aut}(Q)$ is called \textit{the group of inner automorphisms of the quandle $Q$} and is denoted by $\Inn(Q)$. Every automorphism from $\Inn(Q)$ is called \textit{an inner automorphism of the quandle $Q$}. 

For the element $x\in Q$ the subset $\{\varphi(x)~|~\varphi\in \Inn(Q)\}$ of $Q$ is called \textit{the orbit of the element $x$} and is denoted by $Orb(x)$. The set of orbits of elements of a quandle $Q$ is denoted by $Orb(Q)$. A quandle $Q$ is called \textit{connected} if the set $Orb(Q)$ has  only one element, otherwise $Q$ is called \textit{disconnected}. The dihedral quandle $\R_n$ is connected if an only if $n$ is odd. The $n$-th conjugation quandle ${\rm Conj}_n(G)$ of a nontrivial group $G$ is always disconnected.\bigskip
\section{Enveloping groups of quandles}\label{sec3}
For a quandle $Q$ denote by $G_Q$ the group with the set of generators $Q$ and the set of relations $x*y=y^{-1} x y$ fo all $x,y\in Q$. The group $G_Q$ is called \textit{the enveloping group of the quandle $Q$}. For example, if $Q$ is a trivial quandle then the group $G_Q$ is the free abelian group of rank $|Q|$. The enveloping group $G_{\R_3}$ of the dihedral quandle $\R_3$ on $3$ elements is more interesting. It has $3$ generators $x_0, x_1, x_2$ and is defined by the following relations
$$x_0 x_1 x_0^{-1}=x_2=x_1 x_0 x_1^{-1},~x_1 x_2 x_1^{-1} =x_0=x_2 x_1 x_2^{-1},~x_0 x_2 x_0^{-1} =x_1=x_2 x_0 x_2^{-1}.
$$
Using the first relation we remove the element $x_2$ from the list of generators and obtain the following relations
$$x_0 x_1 x_0^{-1}= x_1 x_0 x_1^{-1},~x_0 x_1 x_0^{-1}= x_1^{-1}x_0x_1,~x_1 x_0 x_1^{-1} = x_0^{-1}x_1x_0.$$
The first and the second relations together imply the equality $x_1^2x_0x_1^{-2}=x_0$, i.~e. $x_1^2$ belongs to the center of $G_{\R_3}$. Analogically, the first and the third relation imply that $x_0^2$ belongs to the center of $G_{\R_3}$. Therefore looking at the square of the first relation we have $x_0^2=x_1^2$, and then all the relation are equivalent to the relation $x_0x_1x_0=x_1x_0x_1$. So, we proved the following.
\begin{proposition}\label{dihh}The enveloping group 
$G_{\R_3}$ of the dihedral quandle $\R_3$ has two generators $x_0, x_1$ and two relations $x_1 x_0 x_1 = x_0 x_1 x_0$, $x_0^2 = x_1^2$. In particular, $G_{\R_3}$ is a homomorphic image of the braid group $B_3 = \langle \sigma_1, \sigma_2~|~\sigma_1 \sigma_2 \sigma_1 = \sigma_2 \sigma_1 \sigma_2 \rangle$  on $3$ strands.
\end{proposition}

The subgroup $P_3=\langle \sigma_1^2, \sigma_2^2, \sigma_2 \sigma_1^2 \sigma_2^{-1}\rangle$ of the braid group $B_3$ is called \textit{the pure braid group on $3$ strands}. The group $P_3$ is normal in $B_3$ and we have the following short exact sequence of groups
$$
1 \to P_3 \to B_3 \to \Sigma_3 \to 1.
$$
Defining the homomorphism $B_3 \to G_{\R_3}$ by the rule $\sigma_1 \mapsto x_0$, $\sigma_2 \mapsto x_1$, we have the following short exact sequence of groups
$$
1 \to \mathbb{Z} \to G_{\R_3} \to \Sigma_3 \to 1,
$$
where $\mathbb{Z} = \langle x_0^2\rangle$ is the center of $G_{\R_3}$. This extension does not split. In \cite[Lemma 2.3]{GarVen} it is proved that the enveloping group $G_Q$ of a connected quandle $Q$ can be decomposed $G_Q=[G_Q,G_Q]\rtimes \mathbb{Z}$. This result emplies the following proposition which gives a split decomposition of the group $G_{\R_3}$.
\begin{proposition}The enveloping group
$G_{\R_3}$ is the semi-direct product $G_{\R_3} = \mathbb{Z}_3 \rtimes \mathbb{Z}$.
\end{proposition}
\begin{proof} Using Proposition \ref{dihh} and direct calculations it is easy to check that $[G_{\R_3},G_{\R_3}]=\mathbb{Z}_3$.
\end{proof}

In the general case we have the following proposition about enveloping groups.
\begin{proposition}\label{fingen} Let $Q$ be a quandle.
\begin{enumerate}
\item If $Q$ has finitely many orbits, then the abelianization of the enveloping group $G_{Q}$ is isomorphic to $\mathbb{Z}^{\left|Orb(Q)\right|}$.
\item If $Q$ is finite, then the derived subgroup of the enveloping group $G_Q$  is finitely generated.
\end{enumerate}
\end{proposition}
\begin{proof}
(1) Let $Q = \{ x_i~|~i\in I\}$, then $G_Q$ is generated by elements $\{ a_i~|~i\in I \}$ and is defined by the relations $a_i^{-1} a_j a_i = a_{s(i,j)}$,
where the index $s(i,j)$ is uniquely determined by the multiplication in $Q$
$$x_i * x_j = x_{s(i,j)}.$$ 
Modulo $G_Q^{\prime}$ the relation $a_i^{-1} a_j a_i = a_{s(i,j)}$ has the form $a_j= a_{s(i,j)}$ and we have an isomorphism between $G_Q^{ab}$ and $\mathbb{Z}^{|Orb(Q)|}$.

(2) The commutator subgroup $G_Q^{\prime}$ is generated by commutators $a_i^{-1} a_j^{-1} a_i a_j = a_{s(i,j)}^{-1} a_j$ for all $i,j$. The number of these commutators is finite.
\end{proof}
The dihedral quandle $\R_3$ on $3$ elements is connected, i.~e. it has only $1$ orbit. Therefore by Proposition \ref{fingen} the abelianization of the enveloping group $G_{\R_3}$ is isomorphic to $\mathbb{Z}$. The same result follows from Proposition \ref{dihh} since the enveloping group of the dihedral quandle $\R_3$  has the presentation $G_{\R_3}=\langle x_0, x_1 ~|~x_1 x_0 x_1 = x_0 x_1 x_0, x_0^2 = x_1^2\rangle$.

The natural map $Q\to G_Q$ which maps an element from $Q$ to the corresponding generator of $G_Q$ induces a homomorphism of quandles $Q\to {\rm Conj}(G_Q)$. This homomorphism is not necessarily injective. For example, in \cite[Section 6]{Joyce} Joyce noticed that if $Q$ is a quandle with $3$ elements $x,y,z$ and operation $x*y=z$, $x*z=x$, $z*x=z$, $z*y=x$, $y*x=y$, $y*z=y$, then the map $Q\to G_Q$ maps $x$ and $z$ to the same element and therefore it is not injective.

The following result shows a relation between the automorphism group of $Q$ and the automorphism group of $G_Q$ in the case when the natural map $Q\to G_Q$ is injective.
\begin{proposition}\label{inclus} Let $Q$ be a quandle such that the natural map $Q\to G_Q$ which maps an element from $Q$ to the correspoding generator of $G_Q$ is injective. If $\Aut_Q(G_Q)$ denotes the subgroup of $\Aut(G_Q)$ consisting of automorphisms which keep $Q$ invariant, then the natural map $\phi:\Aut_Q(G_Q) \to \Aut(Q)$ given by $\phi(f) = f|_Q$ is an isomorphism of groups.
\end{proposition}

\begin{proof}
The map $ f|_Q:Q \to Q$ is obviously a bijection of $Q$ for each $f \in \Aut_Q(G_Q)$. Moreover, if $x, y \in Q$, then $f|_Q(x*y)=f|_Q(y^{-1}xy)= f(y)^{-1}f(x)f(y)=f|_Q(x)*f|_Q(y)$, i.~e. $\phi$ is indeed the map from $\Aut_Q(G_Q)$ to $\Aut(Q)$, which is obviously a homomorphism. So, we need to prove that $f$ is a bijection.

If $\phi(f)=\phi(g)$ for $f,g\in \Aut(G_Q)$, then $f(x)=g(x)$ for all $x\in Q$ and then $f=g$ since $G_Q$ is generated by the set $Q$. So, $\phi$ is an injective map. For the surjectivity, let $\alpha \in \Aut(Q)$. Then the composition of $\alpha$ with the natural inclusion  $i:Q \hookrightarrow G_Q$ can be extended to a group homomorphism $\alpha':F(Q) \to G_Q$, where $F(Q)$ is the free group on $Q$. For the elements $x, y \in Q$ we have
$$
\alpha'(x*y) =  i \alpha (x*y)
= i \big(\alpha(x)* \alpha(y) \big)
= \alpha(y)^{-1}\alpha(x) \alpha(y)
= \alpha'(y)^{-1}\alpha'(x) \alpha'(y)
= \alpha'(y^{-1} x y),
$$
i.~e. $\alpha'$ induces a group homomorphism $\tilde{\alpha}: G_Q \to G_Q$. It is easy to see that $\tilde{\alpha}$ is an automorphism of $G_Q$ which keeps $Q$ ivariant with inverse $\widetilde{\alpha^{-1}}$ and $\phi(\tilde{\alpha})= \alpha$.
\end{proof}

Unfortunately, it is quite difficult to check if the natural map $Q\to G_Q$ is injective or not. However, if it is, then by Proposition \ref{inclus} the group ${\rm Aut}(Q)$ is a subgroup of the group $\Aut(G_Q)$. This subgroup is not necessarily normal in $\Aut(G_Q)$. Indeed, if $Q=\{x_1,x_2,\dots,x_n\}$ is the trivial quandle, then the natural map $Q\to G_Q$ is injective and the enveloping group $G_Q$ is the free abelian group $\mathbb{Z}^{n}$ of rank $n$.
 The group ${\rm Aut}(Q)$ is isomorphic to $\Sigma_{n}$ and consists of all possible permutations of the elements from $Q$. The group $\Aut(G_Q)$ is the general linear group $\mathrm{GL}_{n} (\mathbb{Z})$ which acts naturally on $G_Q=\mathbb{Z}^{n}$. 
If $Q$ contains  at least $2$ elements, then denote by
$$
\varphi :\begin{cases}x_1\mapsto x_2,&\\
x_2\mapsto x_1,&\\
x_i\mapsto x_i,&i>2.
\end{cases}\psi: \begin{cases}
x_1\mapsto x_1+x_2,&\\
x_i\mapsto x_i,&i>1.
\end{cases}
$$
The map $\varphi$ belongs to the group $\Aut(Q)=\Sigma_n\leq{\rm GL}_{n}(\mathbb{Z})$ and the map $\psi$ belongs to the group $\Aut(G_Q)={\rm GL}_{n}(\mathbb{Z})$. Since the automorphism $\psi^{-1}\varphi\psi$ of $G_Q$ has the form
$$
\psi^{-1}\varphi\psi : \begin{cases}
x_1 \mapsto x_1,& \\
x_2 \mapsto x_1 - x_2,& \\
x_i \mapsto x_i,& i > 2,
\end{cases}
$$
it does not belong to $\Aut(Q)$, therefore $\Aut(Q)$ is not a normal subgroup of $\Aut(G_Q)$.\bigskip

\section{Automorphisms of conjugation quandles}\label{sec4}
Let $G$ be a group and $\varphi$ be an automorphism of $G$. Since $\varphi\left(x^{y^n}\right)=\varphi(x)^{\varphi(y)^n}$, every automorphism of $G$ induces the automorphism of ${\rm Conj}_n(G)$ and therefore $\Aut(G) \leq \Aut\left({\rm Conj}_n (G)\right)$. Of course, the group $\Aut(G)$ is not necessarily normal in  $\Aut\left({\rm Conj}_n (G)\right)$. For example, if $G$ is a cyclic group of prime order $p\geq5$, then the order of ${\rm Aut}(G)$ is equal to $p-1$. The quandle ${\rm Conj}(G)$ is trivial and $\Aut\left({\rm Conj}(G)\right)=\Sigma_p$. The only proper normal subgroup of $\Sigma_p$ for $p\geq5$ is the alternating group $A_p$ which has order $p!/2>p-1$. Therefore $\Aut(G)$ cannot be normal in $\Aut\left({\rm Conj}(G)\right)$.

We will show
that $\Aut(G)=\Aut\left({\rm Conj}(G) \right)$ if and only if $\Z(G)=1$.

\begin{lemma}
Let $G$ be a group, $Q={\rm Conj}(G)$ and $\varphi$ be an automorphism of $Q$. Then for all $x,y\in Q$ there exists an element $z\in \Z(G)$ such that $\varphi(xy)=\varphi(x)\varphi(y)z$.
\end{lemma}
\begin{proof} For all $x,y,z\in Q$ we have $\varphi(x^{yz})=\varphi(x*{yz})=\varphi(x)*{\varphi(yz)}=\varphi(x)^{\varphi(yz)}$. On the other side 
\begin{multline*}\varphi(x^{yz})=\varphi((x^y)^z)=\varphi(x^y*z)=\varphi(x^y)*{\varphi(z)}=\\=\varphi(x*y)*{\varphi(z)}=(\varphi(x)*\varphi(y))*{\varphi(z)}=\varphi(x)^{\varphi(y)\varphi(z)}
\end{multline*}
 and we have $\varphi(x)^{\varphi(yz)}=\varphi(x)^{\varphi(y)\varphi(z)}$.  
Since $x$ is an arbitrary element of the group $G$, the element $\varphi(yz)(\varphi(y)\varphi(z))^{-1}$ must belong to $\Z(G)$.
\end{proof}
\begin{corollary}
If $G$ is a group with trivial center, then $\Aut(G)=\Aut \left({\rm Conj}(G)\right)$.
\end{corollary}
\begin{proposition}\label{left}
If $G$ is a group with nontrivial center, then $\Aut(G)\neq\Aut \left({\rm Conj} (G)\right)$.
\end{proposition}
\begin{proof}
Let $1\neq x\in \Z(G)$ and $\varphi$ denotes the bijection of $Q=\Conj(G)$ which fixes all elements of $Q$ except $1$ and $x$ and such that $\varphi(1)=x$, $\varphi(x)=1$. Let us show that $\varphi\in {\rm Aut}(Q)$. Let $y,z$ be some elements from $Q$. If $y\in\{1,x\}$, then $\varphi(y*z)=\varphi(y^z)=\varphi(y)=\varphi(y)*{\varphi(z)}$ since $y, \varphi(y)\in\Z(G)$. If $y\notin\{1,x\}$, then $y^z\notin\{1,x\}$, and hence  $\varphi(y*z)=\varphi(y^z)=y^z=\varphi(y)^{\varphi(z)}=\varphi(y)*\varphi(z)$, i.~e. $\varphi$ is an automorphism of $Q$. Since $\varphi(1)\neq1$, the map $\varphi$ does not belong to $\Aut (G)$.
\end{proof}
\begin{corollary}\label{equal}
Let $G$ be a group. Then $\Aut(G)=\Aut \left({\rm Conj} (G)\right)$ if and only if $\Z(G)=1$.
\end{corollary}
In \cite[Proposition 4.7]{BDS} it is proved that for a group $G$ there exists an embedding of groups $\Z(G)\rtimes \Aut (G)\hookrightarrow\Aut\left({\rm Conj} (G)\right)$. We are going to prove that if $G$ is a non-abelian group, then $\Aut (G) \times  \Sigma_{|{\rm Z}(G)|}$ is also imbedded into $\Aut\left({\rm Conj} (G)\right)$.

\begin{proposition}\label{direct}Let $G$ be a non-abelian group. Then the group  $\Aut \left({\rm Conj}(G) \right)$ contains a subgroup isomorphic to the direct product $\Aut (G) \times  \Sigma_{|{\rm Z}(G)|}$.
\end{proposition}

\begin{proof} Consider the map $f:\Aut(G)\to\Aut({\rm Conj}(G))$ which maps every automorphism $\varphi$ of the group $G$ to the map $f(\varphi):\Conj (G) \to \Conj (G)$ which acts on $\Conj (G)$ by the following rule.
$$f(\varphi)(x)=\begin{cases}
\varphi(x),&x\notin \Z(G),\\
x,& x\in \Z(G).
\end{cases}$$
Since $\varphi\in \Aut (G)$, the map $\varphi$ maps $\Z(G)$ to $\Z(G)$ and it maps $G\setminus \Z(G)$ to $G\setminus \Z(G)$. Therefore $f(\varphi)$ is a bijection on $\Conj (G)$. It is easy to show that $f(\varphi)$ preserves the operation of $\Conj (G)$, and hence it is an automorphism of $\Conj (G)$.

For automorphisms $\varphi,\psi\in \Aut (G)$ and element $x\in \Z(G)$ we have $f(\varphi\psi)(x)=x=f(\varphi)f(\psi)(x)$. If $x\notin \Z(G)$, then $\psi(x)\notin \Z(G)$, and hence $f(\varphi\psi)(x)=\varphi(\psi(x))=f(\varphi)(f(\psi)(x))=f(\varphi)f(\psi)(x)$. Therefore $f$ is a homomorphism from $\Aut (G)$ to $\Aut \left({\rm Conj} (G)\right)$. If $f(\varphi)=f(\psi)$, then for all $x\notin \Z(G)$ we have $\varphi(x)=f(\varphi)(x)=f(\psi)(x)=\psi(x)$. In particular, if $x \in G\setminus\Z(G)\neq\varnothing$ and $y\in \Z(G)$, then $\varphi(x)=\psi(x)$, $\varphi(xy)=\psi(xy)$. Therefore $\varphi(y)=\psi(y)$ and $\varphi=\psi$. Hence $f$ is an injective homomorphism from ${\rm Aut}(G)$ to $\Aut({\rm Conj}(G))$.

Consider the map $g:\Sigma_{|{\rm Z}(G)|}\to\Aut({\rm Conj}(G))$ which maps every permutation $\sigma$ of the elements from $\Z(G)$ to the map $g(\sigma)$ which acts on $\Conj (G)$ by the following rule.
$$g(\sigma)(x)=\begin{cases}
x,&x\notin \Z(G),\\
\sigma(x),& x\in \Z(G).
\end{cases}$$
Similarly to the case of the map $f$ it is easy to see that $g$ is an injective homomorphism from $\Sigma_{|{\rm Z}(G)|}$ to $\Aut\left({\rm Conj} (G) \right)$. Let $H$ be a subgroup of $\Aut (\Conj (G))$ generated by $H_1=f(\Aut (G))$ and $H_2=g( \Sigma_{|{\rm Z}(G)|})$. Both groups $H_1,H_2$ are obviously normal in $H$ (moreover, $H_2$ is normal in $\Aut({\rm Conj}(G))$) and $H_1\cap H_2=1$, therefore $H=H_1\times H_2={\rm Aut}(G)\times \Sigma_{|{\rm Z}(G)|}$.
\end{proof}
\begin{corollary} Let $G$ be a group. Then the group $\Aut \left({\rm Conj} (G)\right)$ contains a  subgroup isomorphic to the group
$\Z(G)$.
\end{corollary}

\begin{proof} The group $\Z(G)$ acts faithfully on $\Z(G)$ by right multiplications and therefore $\Z(G)$ is a subgroup of $\Sigma_{|{\rm Z}(G)|}$. If $G$ is abelian then ${\rm Conj}(G)$ is  the trivial quandle and  ${\rm Aut}({\rm Conj}(G))$ is isomorphic to $\Sigma_{|{\rm Z}(G)|}$. If $G$ is not abelian, then by Proposotion \ref{direct} the group $\Sigma_{|{\rm Z}(G)|}$ is a subgroup of $\Aut \left({\rm Conj} (G)\right)$.
\end{proof}
The following result says that if $G\neq1$ is a finite abelian group, then the conclusion of Proposition \ref{direct} is correct if and only if $G=\mathbb{Z}_2$.
\begin{proposition}\label{abbb} Let $G$ be a finite abelian group. Then ${\rm Aut}({\rm Conj}(G))$ contains a subgroup isomorphic to ${\rm Aut}(G)\times\Sigma_{|{\rm Z}(G)|}$ if and only if $G=1$ or $G=\mathbb{Z}_2$.
\end{proposition}
\begin{proof} If $G=1$ or $G=\mathbb{Z}_2$, then ${\rm Aut}(G)=1$ and ${\rm Aut}({\rm Conj}(G))={\rm Aut}(G)\times\Sigma_{|{\rm Z}(G)|}$. Conversely, since $G$ is an abelian group, the quandle ${\rm Conj}(G)$ is the trivial quandle and ${\rm Aut}({\rm Conj}(G))$ is the full permutation group $\Sigma_{|G|}$ which has the order $|G|!$. The group ${\rm Aut}(G)\times \Sigma_{|{\rm Z}(G)|}$ is equal to the group ${{\rm Aut}(G)}\times\Sigma_{|G|}$ and has order $|{\rm Aut}(G)|\cdot |G|!$. Therefore $|{\rm Aut}(G)|=1$ and $G=1$ or $G=\mathbb{Z}_2$. 
\end{proof} 
 
 The following result describes all finite groups for which ${\rm Aut}({\rm Conj}(G))={\rm Aut}(G)\times \Sigma_{|{\rm Z}(G)|}$.
\begin{theorem}
Let $G$ be a finite group. Then ${\rm Aut}({\rm Conj}(G))={\rm Aut}(G)\times \Sigma_{|{\rm Z}(G)|}$ if and only if $\Z(G)=1$ or $G=\mathbb{Z}_2$.
\end{theorem}
\begin{proof}Corollary \ref{equal} and Proposition \ref{abbb} imply that if $\Z(G)=1$ or $G=\mathbb{Z}_2$, then ${\rm Aut}({\rm Conj}(G))={\rm Aut}(G)\times \Sigma_{|{\rm Z}(G)|}$ and we need to prove that groups with trivial center and $\mathbb{Z}_2$ are all groups for which ${\rm Aut}({\rm Conj}(G))={\rm Aut}(G)\times \Sigma_{|{\rm Z}(G)|}$.

Suppose that $\Z(G)\neq 1$. Then according to Proposition \ref{direct} the group ${\rm Aut}(G)\times \Sigma_{|{\rm Z}(G)|}$ consists of all automorphisms of the form
\begin{equation}\label{gfo}
x\mapsto\begin{cases}
\varphi(x),&x\notin \Z(G),\\
\sigma(x),& x\in \Z(G),
\end{cases}
\end{equation}
where $\varphi$ is an automorphism of $G$ and $\sigma$ is a permutation of elements from $\Z(G)$. Following \cite{BDS}, for an element $a\in \Z(G)$ denote by $t_a$ the automorphism of the quandle ${\rm Conj}(G)$ of the form $t_a:x\mapsto xa$. 

Since $\Z(G)\neq 1$, for an element $1\neq a\in \Z(G)$ the map $t_a:x\mapsto xa$ must have the form (\ref{gfo}). If in $G$ there exist two elements $x,y\notin\Z(G)$ such that $xy\notin\Z(G)$, then from equality (\ref{gfo}) we have $\varphi(x)\varphi(y)=\varphi(xy)=t_a(xy)=axy$. On the other hand  from (\ref{gfo}) we have $\varphi(x)=t_a(x)=ax$ and $\varphi(y)=t_a(y)=ay$. Therefore $a=1$, what contradicts the choice of $a$ and we proved that for all elements $x,y\notin\Z(G)$ the product $xy$ belongs to $\Z(G)$.
It means that $|G:\Z(G)|\leq2$ and $G$ is abelian. Therefore by Proposition \ref{abbb} we have $G=\mathbb{Z}_2$.
\end{proof}

As we already mentioned, in \cite[Proposition 4.7]{BDS} it is proved that the group $\Aut\left({\rm Conj} (G)\right)$ contains a subgroup isomorphic to the semidirect product $\Z(G)\rtimes \Aut (G)$. The following theorem describes all finite groups which satisfy the equality $\Aut\left({\rm Conj} (G)\right)=\Z(G)\rtimes \Aut (G)$, what gives the answer to \cite[Problem 4.8]{BDS}.
\begin{theorem}\label{solution}
Let $G$ be a finite group. Then $\Aut \left({\rm Conj} (G)\right)=\Z(G)\rtimes \Aut(G)$ if and only if either $\Z(G)=1$ or $G$ is one of the groups  ${\mathbb{Z}_2}$, $\mathbb{Z}_2^2$ or $\mathbb{Z}_3$.
\end{theorem}

\begin{proof} The if-part of the theorem is simple.
By Corollary \ref{equal} if the center of the group $G$ is trivial, then $\Aut \left({\rm Conj} (G)\right)=\Aut (G)=\Z(G)\rtimes \Aut (G)$. For groups $\mathbb{Z}_2$, $\mathbb{Z}_2^2$ and $\mathbb{Z}_3$ we have the following equalities
\begin{align}
\notag {\rm Aut}({\rm Conj}(\mathbb{Z}_2))&=\Sigma_2=\mathbb{Z}_2=\mathbb{Z}_2\rtimes {\rm Aut}(\mathbb{Z}_2),\\
\notag {\rm Aut}({\rm Conj}(\mathbb{Z}_2^2))&=\Sigma_4=\mathbb{Z}_2^2\rtimes\Sigma_3=\mathbb{Z}_2^2\rtimes{\rm Aut}(\mathbb{Z}_2^2),\\
\notag {\rm Aut}({\rm Conj}(\mathbb{Z}_3))&=\Sigma_3=\mathbb{Z}_3\rtimes\Sigma_2=\mathbb{Z}_3\rtimes{\rm Aut}(\mathbb{Z}_3),
\end{align}
 and we need to prove that groups without center and the groups ${\mathbb{Z}_2}$, $\mathbb{Z}_2^2$ and $\mathbb{Z}_3$ are all the groups for which  $\Aut \left({\rm Conj} (G)\right)=\Z(G)\rtimes  \Aut (G)$.

For the element $a\in \Z(G)$ denote by $t_a$ the automorphism of the quandle ${\rm Conj}(G)$ of the form $t_a:x\mapsto xa$. The group generated by $t_a$ for all $a\in \Z(G)$ is isomorphic to $\Z(G)$ and the subgroup of $\Aut({\rm Conj}(G))$ which is isomorphic to $\Z(G)\rtimes  \Aut (G)$ consists of the automorphisms of the form $t_a\varphi$ for $a\in \Z(G)$, $\varphi\in{\rm Aut}(G)$ (see \cite[Proposition 4.7]{BDS} for details). So, we need to prove that if every automorphism of ${\rm Conj}(G)$ has the form $t_a\varphi$ for some $a\in \Z(G)$ and $\varphi\in{\rm Aut}(G)$, then either  $\Z(G)=1$ or $G$ is one of the groups $\mathbb{Z}_2$, $\mathbb{Z}_2^2$ or $\mathbb{Z}_3$. 

Suppose that $\Z(G)\neq1$. For a permutation $\sigma$ of the elements from $\Z(G)$ denote by $f$ the automorphism of $\Conj(G)$ of the form
$$f(x)=\begin{cases}
x,&x\notin \Z(G),\\
\sigma(x),& x\in \Z(G).
\end{cases}$$
Suppose that $f=t_a\varphi$ for $a\in \Z(G)$ and $\varphi\in \Aut (G)$. If in $G$ there exist two elements $x,y\notin \Z(G)$ such that $xy\notin \Z(G)$, then
\begin{eqnarray}
\notag x&=&f(x)=t_a\varphi(x)=a\varphi(x),\\
\notag y&=&f(y)=t_a\varphi(y)=a\varphi(y),\\
\notag xy&=&f(xy)=t_a\varphi(xy)=a\varphi(x)\varphi(y).
\end{eqnarray}
These equalities imply that $a=1$ and $f=\varphi$. Since $\Z(G)\neq 1$, we can choose $\sigma$ such that $\sigma(1)\neq1$ and $f$ cannot belong to ${\rm Aut}(G)$. Therefore for every two elements $x,y\notin \Z(G)$ the product $xy$ belongs to $\Z(G)$. It means that $|G:\Z(G)|\leq 2$, therefore $G$ is abelian and $|{\rm Aut}({\rm Conj}(G))|=|\Sigma_{|G|}|=|G|!$. Since every automorphism of $G$ fixes the unit element, we have $|{\rm Aut}(G)|\leq(|G|-1)!$ and since $\Aut \left({\rm Conj} (G)\right)=\Z(G)\rtimes \Aut(G)=G\rtimes \Aut(G)$, we have the equality $|{\rm Aut}(G)|=(|G|-1)!$, i.~e. every permutation of nontrivial elements from $G$ is an automorphism of $G$, in particular, all nontrivial elements of $G$ have the same order. 

Suppose that in $G$ there exist at least $5$ distinct elements $\{1,x,y,xy,t\}$. Since every permutation of elements from $G$ which fixes $1$ is an automorphism of $G$, the map $\varphi(1)=1$, $\varphi(x)=x$, $\varphi(y)=y$, $\varphi(xy)=t$ can be extended to the automorphism of $G$. It means, that $t=\varphi(xy)=\varphi(x)\varphi(y)=xy$, what contradicts the fact that the elements $\{1,x,y,xy,t\}$ are distinct. Therefore the order of $G$ is less then or equal to $4$ and since all nontrivial elements of $G$ have the same order, $G$ is one of the groups $\mathbb{Z}_2$, $\mathbb{Z}_2^2$ or $\mathbb{Z}_3$.
\end{proof}\bigskip
\section{Automorphisms of core quandles and Takasaki quandles}\label{sec5}
 The following proposition is an analogue of \cite[Proposition 4.7]{BDS} for core quandles of groups.
\begin{proposition}\label{coreaut}Let $G$ be a group with the center $\Z(G)$. Then ${\rm Aut}({\rm Core}(G)) $ has a subgroup isomorphic to the group $\Z(G)\rtimes{\rm Aut}(G)$.
\end{proposition}
\begin{proof} Denote by $*$ the operation in the quandle ${\rm Core}(G)$, i.~e. $x*y=yx^{-1}y$ for all $x,y\in G$. If $\varphi$ is an automorphism of $G$, then $\varphi(yx^{-1}y)=\varphi(y)\varphi(x)^{-1}\varphi(y)$. Therefore $\varphi(x*y)=\varphi(x)*\varphi(y)$ and $\varphi$ induces the automorphism of $\Core(G)$. Denote by $H_1$ the subgroup of $\Core(G)$ generated by all automorphisms induced by automorphisms from ${\rm Aut}(G)$. For an element $a\in \Z(G)$ denote by $t_a$ the bijection of $\Core(G)$ of the form $t_a:x\mapsto ax$, which is obviously an automorphism of $\Core(G)$. Denote by $H_2$ the subgroup of $\Core(G)$ generated by $t_a$ for all $a\in \Z(G)$. Since  $\varphi(1)=1$ for every automorphism $\varphi\in H_1$, the intersection of $H_1$ and $H_2$ is trivial. For the automorphism $\varphi\in H_1$ and the automorphism $t_a\in H_2$ we have $\varphi t_a\varphi^{-1}=t_{\varphi(a)}$, therefore $H_2$ is the normal subgroup of the group $H$ generated by $H_1, H_2$. Hence $H$ is the semidirect product $H_2\rtimes H_1$ of $H_2=\Z(G)$ and $H_1={\rm Aut}(G)$.
\end{proof}

If $G$ is an abelian group, then the core quandle $\Core(G)$ is the Takasaki quandle $T(G)$ and Proposition \ref{coreaut} has strong modification which is proved in  \cite[Theorem 4.2]{BDS}.

\begin{theorem}\label{thm-dbs} If $G$ is an additive abelian group without $2$-torsion, then
\begin{enumerate}
\item $\Aut\left(T(G)\right)= G \rtimes \Aut(G)$,
\item $\Inn\left(T(G)\right)= 2G \rtimes \mathbb{Z}_2$.
\end{enumerate}
\end{theorem}

It is obvious that if $G = \mathbb{Z}_2^n$ is an elementary abelian $2$-group, then the Takasaki quandle $T(G)$ is the trivial quandle on $2^n$ elements. The group of automorphisms of this quandle is isomorphic to $\Sigma_{2^n}$ and the group of inner automorphisms is trivial. Here we attempt to complement Theorem \ref{thm-dbs} considering finite abelian groups with $2$-torsion.

Let $M = (m_{ij})$ be a symmetric $n \times n$ matrix with entries from $\mathbb{N} \cup\{ \infty\}$ such that $m_{ii} = 1$ for all $i=1,\dots,n$ and $m_{ij} > 1$ for $i \neq j$. \textit{The Coxeter group $W(M)$ of type $M$} is the group which has the following presentation
$$W(M) =\{a_1,\dots , a_n~ |~ (a_i a_j)^{m_{ij}} = 1~\textrm{for all}~ 1 \leq i, j \le n\}.$$
More results about Coxeter groups can be found, for example, in \cite{Cox}.
\begin{theorem}\label{even}
Let $G$ be a finite additive abelian group with cyclic decomposition
$$G= \mathbb{Z}_{n_1} \oplus \cdots \oplus \mathbb{Z}_{n_r} \oplus \mathbb{Z}_{n_{r+1}} \oplus \cdots \oplus \mathbb{Z}_{n_k}$$ such that
$n_i$ divides $n_{i+1}$ for all $i=1,\dots k-1$ and $n_k>2$ is even and $N$ be the number
$$N=\begin{cases}
\prod_{j=1}^r n_j \cdot \prod_{i=r+1}^k n_i/2,& \mathbb{Z}_{n_r}~\textrm{is the largest odd order component},\\
\prod_{i=1}^k n_i/2  & \textrm{if there is no odd order component}.
\end{cases}$$
Then the group $\Inn\left(T(G) \right)$ is isomorphic to the Coxeter group $W(M)$, where $M = (m_{ij})$ is the symmetric $N \times N$ matrix with $m_{ij} =n_k/2$ for all $i \neq j$.
\end{theorem}
\begin{proof}
By the definition $\Inn\big(T(G) \big)= \langle S_x~|~x \in T(G) \rangle$ and it is easy to see that $S_x^2= \id_{T(G)}$ for all $x \in T(G)$. Note that the exponent of $G$ is $n_k>2$. For the elements $x, y \in T(G)$ using direct calculations we have
$$(S_xS_y)^{n_k/2}(z)=n_k(x-y)+z=z$$
 for all $z \in T(G)$. Therefore the group  $\Inn\big(T(G) \big)$ is generated by involutions $S_x$ subject to the relations $(S_xS_y)^{n_k/2}= \id_{T(G)}$ for all $x, y \in T(G)$. In order to determine distinct generators $S_x$ of $\Inn(G)$, notice that $S_x=S_y$ if and only if $2x=2y$. Let $x=(x_1,\dots, x_r, x_{r+1}, \dots, x_k)$ and $y=(y_1,\dots, y_r,y_{r+1},  \dots, y_k)$ be two elements of $G$. Then $2x=2y$ if and only if $x_j=y_j$ for all $j=1,\dots r$ and $x_i=y_i\pm n$ for all $i=r+1,\dots,k$. Thus the number $N$ of distinct involutions $S_x$ in $\Inn(G)$ is equal to $\prod_{j=1}^r n_j \cdot \prod_{i=r+1}^k n_i/2$ if $\mathbb{Z}_{n_r}$ is the largest odd order component in $G$ and $\prod_{i=1}^k n_i/2$ if there are no odd order components. This proves that $\Inn\big(T(G) \big)$ is isomorphic to $W(M)$, where $M = (m_{ij})$ is an $N \times N$ matrix with $m_{ij} =n_k/2$ for all $i \neq j$.
\end{proof}

For an arbitrary group $G$ a discription of the group of inner automorphisms $\Inn(\Core(G))$ as a quotient of some subgroup of the group $(G\times G)\rtimes {\mathbb{Z}}_2$ is presented in \cite[Proposition 4.14]{Nos}. Theorem \ref{even} and Theorem \ref{thm-dbs}(2) give alternative description of $\Inn(\Core(G))$ in the case when $G$ is an abelian group.

If $G=\mathbb{Z}_n$ is the cyclic group of order $n$, then the Takasaki quandle $T(G)$ is the dihedral quandle $\R_n$ and Theorem \ref{even} implies the following result.

\begin{corollary}\label{inn-abelian-dihedral}
Let $n>2$ be an integer. Then the group of inner automorphisms $\Inn(\R_{2n})$ of the dihedral quandle $\R_{2n}$ is isomorphic to the Coxeter group $W(M)$, where $M = (m_{ij})$ is the $n \times n$ matrix with $m_{ij} =n$ for all $i \neq j$.
\end{corollary}

\begin{corollary}\label{inn-aut-abelian}
If $G=\mathbb{Z}_4^k$, then the group of inner automorphisms $\Inn\big(T(G) \big)$ is the elementary abelian group $\mathbb{Z}_2^{2^k}.$
\end{corollary}

\begin{proof}
In this case $N=2^k$ and $S_x S_y=S_y S_x$ for all $x, y \in T(G)$.
\end{proof}
\begin{corollary}\label{r4-auto}
The group $\Aut(\R_4)$  is isomorphic to the semidirect product $(\mathbb{Z}_2 \oplus \mathbb{Z}_2) \rtimes \mathbb{Z}_2$, where $\mathbb{Z}_2$ acts on $\mathbb{Z}_2\oplus\mathbb{Z}_2$ by permuting of the components.
\end{corollary}
\begin{proof}
The quandle $\R_4$ has $4$ elements $a_0, a_1, a_2, a_3$. From Corollary \ref{inn-aut-abelian} the group $\Inn(\R_4 )$ is isomorphic to the group $\mathbb{Z}_2 \oplus \mathbb{Z}_2= \langle S_{a_0} \rangle  \oplus \langle S_{a_1} \rangle$. Let $\phi$ be the following automorphism of $\R_4$
$$
\phi:\begin{cases}
a_0 \mapsto a_1, &  \\
a_1 \mapsto a_0, &  \\
a_2 \mapsto a_3, &  \\
a_3 \mapsto a_2. &
\end{cases} 
$$
It is easy to see that $\phi$ is not  an inner automorphism of $\R_4$ and is of order $2$. Moreover, $\phi S_{a_0} \phi^{-1}=S_{a_1}$ and
$\phi S_{a_1} \phi^{-1}=S_{a_0}$. Therefore the group $H=\left( \langle S_{a_0} \rangle  \oplus \langle S_{a_1} \rangle \right)\rtimes  \langle\phi \rangle$ is a subgroup of $\Aut(\R_4)$. For the symmetric group $\Sigma_4$ we can write $\Sigma_4= H \cup \sigma_1H \cup \sigma_2H$, where
$$
\sigma_1: \begin{cases}
a_0 \mapsto a_1, &  \\
a_1 \mapsto a_0, &  \\
a_2 \mapsto a_2, &  \\
a_3 \mapsto a_3, &
\end{cases}
\sigma_2:\begin{cases}
a_0 \mapsto a_3, &  \\
a_1 \mapsto a_1, &  \\
a_2 \mapsto a_2, &  \\
a_3 \mapsto a_0. &
\end{cases}
$$
Since $\sigma_1, \sigma_2 \not\in \Aut(\R_4)$, we have $\Aut(\R_4) \cong H \cong  (\mathbb{Z}_2 \oplus \mathbb{Z}_2) \rtimes \mathbb{Z}_2$.
\end{proof}

\begin{remark}
The dihedral quandle $\R_{2n}=\{a_0,a_1,\dots,a_{2n-1}\}$ is known to be disconnected with two orbits  $\{a_0, a_2, \dots, a_{2n-2}\}$ and $\{a_1, a_3, \dots, a_{2n-1}\}$. Define the map $\phi: \R_{2n} \to \R_{2n}$ by the rule $\phi(a_i)=a_{i+1}$, $\phi(a_{i+1})=a_i$ for all $i=0,2,\dots, 2n-2$. Direct calculations show that the map $\phi$ is an automorphism of $\R_{2n}$ if and only if $n=1$ or $n=2$. Therefore the arguments from the proof of Proposition \ref{r4-auto} do not work in the case of quandle $\R_{2n}$ for $n >2$. This observation also shows that a bijection of a quandle which permutes the orbits of the quandle is not neccessarily an automorphism.
\end{remark}\bigskip
\section{Quandles with $k$-transitive action of the automorphism group}\label{sec6}
Let a group $G$ acts on a sex $X$ from the left, i.~e. there exists a homomorphism $G\to \Sigma_{|X|}$ which maps an element $g\in G$ to the permutation $x\mapsto gx$ of the set $X$. For the number $k\leq|X|$ we say that $G$ acts $k$-transitively on $X$ if for each pair of $k$-tuples $(x_1,\dots,x_k)$ and $(y_1,\dots,y_k)$ of distinct elements from $X$ there exists an element $g\in G$ such that $gx_i = y_i$ for $i=1,\dots,k$. For the sake of simplicity, if $k>|X|$, then we say that $G$ acts $k$-transitively on $X$ independently on $G$.

If $Q$ is a quandle, then the group of inner automorphisms $\Inn(Q)$ acts naturally on $Q$. If the group $\Inn(Q)$ acts $k$-transitively on $Q$, then we say that the quandle $Q$ is $k$-transitive. In particular, the quandle $Q$ is $1$-transitive if and only if it is connected. McCarron  proved that if $Q$ is a finite $2\leq k$-transitive quandle with at least four elements, then $k = 2$ \cite[Proposition 5]{McC}. Moreover, the dihedral quandle $\R_3$ on $3$ elements is the only $3$-transitive quandle. Therefore higher order transitivity does not exist in quandles with at least four elements. 
 
We investigate the following problem from \cite[Problem 6.7]{BDS}: For an integer $k\geq 2$ classify all finite quandles $Q$ for which $\Aut(Q)$ acts $k$-transitively on $Q$.
\begin{lemma}\label{center}
Let $Q$ be a quandle and $\varphi$ be an automorphism of  $Q$. Then $\varphi\left(\Z(Q)\right)=\Z(Q)$.
\end{lemma}
\begin{proof}
If $x\in \Z(Q)$, then for all $y\in Q$ we have $\varphi(x)*\varphi(y)=\varphi(x*y)=\varphi(x)$, therefore $\varphi(x)\in \Z(Q)$. Conversely, if $\varphi(x)\in \Z(Q)$, then for all $y\in Q$ we have $\varphi(x*y)=\varphi(x)*{\varphi(y)}=\varphi(x)$. Therefore $x*y=x$ for all $y\in Q$ and $x\in \Z(Q)$.
\end{proof}
\begin{theorem}\label{tran}
Let $Q$ be a finite quandle. Then the following statements are equivalent:
\begin{enumerate}
\item $Q$ is either trivial quandle or $Q=\R_3$.
\item $\Aut(Q)=\Sigma_{|Q|}$. 
\item $\Aut(Q)$ acts $3$-transitively on $Q$.
\end{enumerate}
\end{theorem}
\begin{proof}The implications $(1)\Rightarrow(2)$ and $(2)\Rightarrow(3)$ are obvious and the only thing which we need to prove is that the quandle $Q$ with $3$-transitive action of ${\rm Aut}(Q)$ is either trivial or is isomorphic to the dihedral quandle $\R_3$.

If $\Z(Q)\neq\varnothing$, then since ${\rm Aut}(Q)$ acts $3$-transitively (and therefore $1$-transitively) on $Q$ by Lemma \ref{center} the quandle $Q$ is trivial and the statement is proved. 

If $\Z(Q)=\varnothing$, then $|Q|\geq3$ and for an element $x\in Q$ there exists an element $y\in Q$ such that $z=x*y\neq x$. Since $x\neq y$, we have $z=x*y\neq y*y=y$, therefore $x,y,z$ are three distinct elements. Let $t$ be an arbitrary element of $Q$. If $t\notin\{x,y\}$, then since ${\rm Aut}(Q)$ acts $3$-transitively on $Q$ there exists an automorphism $\varphi$ of $Q$ such that $\varphi(x)=x$, $\varphi(y)=y$ and $\varphi(z)=t$, but $\varphi(z)=\varphi(x*y)=\varphi(x)*\varphi(y)=x*y=z$. Therefore  $Q=\{x,y,z\}$ has only three elements.

Since $\Aut (Q)$ acts $3$-transitively on $Q$, there exists an automorphism $\varphi$ such that $\varphi(x)=y$, $\varphi(y)=x$, $\varphi(z)=z$, therefore $y*x=\varphi(x)*{\varphi(y)}=\varphi(x*y)=\varphi(z)=z=x*y$.
Again since $\Aut (Q)$ acts $3$-transitively on $Q$, there exists an automorphism $\psi$ such that $\psi(x)=z$, $\psi(y)=x$, $\psi(z)=y$ and therefore $y=\psi(z)=\psi(x*y)=\psi(x)*{\psi(y)}=z*x$. Repeating this argument for all the possible triples $(a_1,a_2,a_3)$, where $a_i\in\{x,y,z\}$, we have the following relation in $Q$
$$x*y=y*x=z,~x*z=z*x=y,~y*z=z*y=x,$$
i.~e. $Q=\{x,y,z\}$ is the dihedral quandle $\R_3$.
\end{proof}

\begin{corollary}If $\Aut(Q)$ acts $3$-transitively on $Q$, then  for all $k\geq3$ it acts $k$-transitively on $Q$.
\end{corollary}
The following question for $k=2$ remains open.
\begin{question}Classify all finite quandles $Q$ with $2$-transitive action of $\Aut(Q)$.
\end{question}\bigskip

\section{Automorphisms of  extensions of quandles}\label{sec7}
Let $Q$ be a quandle and $S$ be a set.  A map
$\alpha : Q\times Q \to \Sigma_{|S|}$ is called {\it a constant quandle cocycle} if it satisfies two conditions:
\begin{enumerate}
\item $\alpha(x*y,z)\alpha(x,y) = \alpha(x*z,y*z)\alpha(x,z)$ for all $x, y, z \in Q$,
\item $\alpha(x,x) = \id_S$ for all $x \in Q$.
\end{enumerate}
Examples of constant quandle cocycles can be found, for example, in \cite{Andruskiewitsch}. The set of all constant quandle cocycles $\alpha : Q\times Q \to \Sigma_{|S|}$ is denoted by $\mathcal{C}^2(Q,S)$. For a quandle $Q$, a set $S$ and a constant quandle cocycle $\alpha:Q\times Q\to\Sigma_{|S|}$ consider the set $Q\times S=\{(x,t)~|~x\in Q,t\in S\}$ with the operation $*$ given by the rule
\begin{equation}\label{non-abelian-structure}
(x, t) * (y, s) = \left(x * y, \alpha(x,y)(t)\right)
\end{equation}
for $x, y \in Q$ and $s, t \in S$. The set $Q\times S$ with the operation $*$ is a quandle which is called {\it the non-abelian extension of $Q$ by $S$ with respect to $\alpha$} and is denoted by $Q \times_\alpha S$. Such quandles were introduced in \cite{Andruskiewitsch}. There exists a surjective quandle homomorphism from $Q \times_\alpha S$ onto $Q$, namely, projection onto 
the first component. On the other hand, for each $x \in Q$, the set $\{ (x,t)~|~t \in S \}$ is a trivial subquandle of $Q \times_\alpha S$. 

Two constant quandle cocycles $\alpha, \beta\in \mathcal{C}^2(Q,S)$ are said to be {\it cohomologous} if there exists a map $\lambda:Q \to \Sigma_{|S|}$ such that $\alpha(x,y)=\lambda(x*y)^{-1} \beta(x,y)\lambda(x)$ for all $x,y \in Q$. The relation of being cohomologous is an equivalence relation on $\mathcal{C}^2(Q,S)$. The equivalence class of a constant quandle cocycle $\alpha$ is called \textit{the cohomology class of $\alpha$} and is denoted by $[\alpha]$. The symbol $\mathcal{H}^2(Q,S)$ denotes the set of cohomology classes of constant quandle cocycles. The following result generalizes the result from \cite[Lemma 4.8]{Carter4} formulated for Alexander quandles.

\begin{lemma}\label{cohomologous-isomorphic}Let $Q$ be a quandle and $S$ be a set. If constant quandle cocycles  $\alpha, \beta:Q\times Q\to \Sigma_{|S|}$ are cohomologous, then the quandles $Q \times_\alpha S$ and $Q \times_\beta S$ are isomorphic.
\end{lemma}
\begin{proof} Denote by $*$ the operation in the quandle $Q\times_{\alpha}S$ and by $\circ$ the operation in the quandle $Q\times_{\beta}S$. Since constant quandle cocycles  $\alpha$ and  $\beta$ are cohomologous, there exists a map $\lambda:Q \to \Sigma_{|S|}$ such that $\alpha(x,y)=\lambda(x*y)^{-1} \beta(x,y)\lambda(x)$ for all $x,y \in Q$. Denote by $f$ the map $f: X \times_\alpha S \to X \times_\beta S$ which maps the pair $(x,t)$ to the pair $\left(x, \lambda(x)(t)\right)$. The map $f$ is obviously a bijection between $Q \times_\alpha S$ and $Q \times_\beta S$. Further, for $(x,t), (y,s) \in Q \times_\alpha S$ we have
\begin{align}
\notag f \left((x,t)*(y,s) \right) &= f\left(x*y, \alpha(x,y)(t)\right) =  \left(x*y, \lambda(x*y) \alpha(x,y)(t)\right)\\
\notag&  =  \left(x*y, \beta(x,y)\lambda(x)(t) \right)  =  \left(x,\lambda(x)(t)\right)\circ \left(y,\lambda(y)(s)\right)\\
\notag&  =  f (x,t)\circ f (y,s),
\end{align}
i.~e. $f$ is an isomorphism of quandles.
\end{proof}

\begin{lemma}\label{aut-constant-action}
Let $Q$ be a quandle, $S$ be a set and $\alpha$ be a constant quandle cocycle. Then for $\phi\in \Aut(Q)$ and $\theta\in \Sigma_{|S|}$ the map $^{(\phi, \theta)}:\mathcal{C}^2(Q,S)\to\mathcal{C}^2(Q,S)$ defined by the rule $$^{(\phi,\theta)}\alpha(x,y)=\theta  \alpha \big(\phi^{-1}(x), \phi^{-1}(y) \big) \theta^{-1} $$
gives the left action of $\Aut(Q) \times \Sigma_{|S|}$ on $\mathcal{C}^2(Q,S)$ which induces an action on $\mathcal{H}^2(Q,S)$.
\end{lemma}
\begin{proof} Using direct calculation it is easy to check that $^{(\phi,\theta)}\alpha$ is a constant quandle cocycle and the map $^{(\phi, \theta)}:\mathcal{C}^2(Q,S)\to\mathcal{C}^2(Q,S)$ is the left action of $\Aut(Q)\times\Sigma_{|S|}$ on $\mathcal{C}^2(Q,S)$. If $\alpha,\beta$ are cohomologous constant cocycles, then there exists a map $\lambda:Q \to \Sigma_{|S|}$ such that $\alpha(x,y)= \lambda(x*y)^{-1} \beta(x,y)\lambda(x)$ for all $x, y \in Q$. Acting on this equality by the map $^{(\phi,\theta)}$ we have
$$^{(\phi,\theta)}\alpha(x,y)= {\lambda^{\prime}(x*y)^{-1}} ~^{(\phi,\theta)}\beta(x,y)\lambda^{\prime}(x),$$ 
where the map $\lambda^{\prime}:Q \to\Sigma_{|S|}$ is given by $\lambda'(x)=\theta \lambda(\phi^{-1}(x) )\theta^{-1}$. Therefore $^{(\phi,\theta)}\alpha$ and $^{(\phi,\theta)}\beta$ are cohomologous and the map $(\phi,\theta)$ which maps the class $[\alpha]$ to the class $\left[^{(\phi,\theta)}\alpha\right]$ is a well defined map $\mathcal{H}^2(Q,S)\to\mathcal{H}^2(Q,S)$. 
\end{proof}
If a group $G$ acts on a set $X$ from the left and $x \in X$, then the set $G_x=\{g \in G~|~gx=x \}$ forms a subgroup of $G$ which is called \textit{the stabilizer of $x$}.
By Lemma \ref{aut-constant-action} the group ${\rm Aut}(Q)\times\Sigma_{|S|}$ acts on the set $\mathcal{C}^2(Q,S)$ of all constant quandle cocycles. The following result gives a relation between groups ${\rm Aut}(Q)\times\Sigma_{|S|}$ and ${\rm Aut}(Q\times_{\alpha}S)$.

\begin{theorem}\label{exterrr}
Let $Q$ be a quandle, $S$ be a set and $\alpha:Q\times Q\to \Sigma_{|S|}$ be a constant quandle cocycle. Then there exists an embedding $({\rm Aut}(Q)\times \Sigma_{|S|})_{\alpha}\hookrightarrow {\rm Aut}(Q\times_{\alpha}S)$.
\end{theorem}

\begin{proof}For an automorphism $\phi$ from ${\rm Aut}(Q)$  and a permutation $\theta$ from $\Sigma_{|S|}$ such that $(\phi,\theta)$ belongs to $({\rm Aut}(Q)\times\Sigma_{|S|})_{\alpha}$ define the  map $\gamma: Q \times_\alpha S \to Q \times_\alpha S$ by the rule $\gamma(x, t)= \left(\phi(x), \theta(t)\right)$. 
The map $\gamma$ is obviously a bijection of $Q\times_{\alpha}S$. Since $(\phi,\theta)\in({\rm Aut}(Q)\times\Sigma_{|S|})_{\alpha}$, we have  $^{(\phi, \theta)}\alpha=\alpha$ and by the definition of the map $^{(\phi,\theta)}$ for all $x, y \in Q$ we have the equality $\theta  \alpha \big(\phi^{-1}(x), \phi^{-1}(y) \big) \theta^{-1}=\alpha(x,y)$. For the elements $x, y \in Q$,  $t, s \in S$ we have
\begin{align}
\notag\gamma \left((x,t)*(y,s) \right)& = \gamma\left(\left(x*y, \alpha(x,y)(t)\right)\right)=  \left(\phi(x)*\phi(y), \theta\alpha(x,y)(t)\right)\\
\notag  &= \left(\phi(x)*\phi(y), \alpha(\phi(x), \phi(y))\theta(t) \right)=  \left(\phi(x),\theta(t)\right)*\left(\phi(y),\theta(s)\right)\\
\notag  &=  \gamma (x,t)*\gamma (y,s),
\end{align}
i.~e. $\gamma$ is an automorphism of $Q\times_{\alpha}S$ and we have an injective map
$$\Psi: \big({\rm Aut}(Q) \times \Sigma_{|S|}\big)_{\alpha} \to \Aut(Q \times_\alpha S)$$
which maps a pair $(\phi,\theta)$ to the automorphism $\gamma$.   If $(\phi_1, \theta_1), (\phi_2, \theta_2) \in \big({\rm Aut}(Q) \times \Sigma_{|S|}\big)_{\alpha}$, then
 $$\gamma_1 \gamma_2(x, t)= \gamma_1\big(\phi_2(x), \theta_2(t)\big)=\big(\phi_1\phi_2(x), \theta_1\theta_2(t)\big).$$
Thus $\Psi\big( (\phi_1, \theta_1) (\phi_2, \theta_2) \big)=\Psi\big( (\phi_1, \theta_1)\big) \Psi \big((\phi_2, \theta_2) \big)$, and hence $\Psi$ is a homomorphism.
\end{proof}
\begin{remark}
If $(\phi, \theta)$ is an element from $\Aut(Q) \times \Sigma_{|S|}$ and the map $\gamma: Q \times_\alpha S \to Q \times_\alpha S$ defined by $\gamma(x, t)= \big(\phi(x), \theta(t)\big)$ is an automorphism of $Q\times_{\alpha}S$, then reversing the preceeding calculations it is easy to see that $(\phi, \theta) \in \big({\rm Aut}(Q) \times \Sigma_{|S|}\big)_{\alpha}$, i.~e. the map $\gamma$ defined by $\gamma(x, t)= \big(\phi(x), \theta(t)\big)$ is an automorphism of $Q\times_{\alpha}S$ if an only if $^{(\phi,\theta)}\alpha=\alpha$.
\end{remark}
For an abelian group $A$ denote by $\psi: A \to \Sigma_{|A|}$ the map which maps the element $a\in A$ to the map $\psi_a(b)=b+a$ for all $b \in A$. Following \cite{Carter3} we call the function  $\mu : Q \times Q \to A$   {\it a quandle  $2$-cocycle} if the map $\psi\mu : Q \times Q \to \Sigma_{|A|}$ is a constant quandle cocycle.
Two quandle $2$-cocycles $\mu, \nu: Q \times Q \to A$ are said to be {\it cohomologous} if the corresponding constant quandle cocycles $\psi\mu,\psi\nu$ are homologous. The set of cohomology classes of quandle $2$-cocycles is the second cohomology group  $\Ha^2(Q,A)$ of the quandle $Q$ with coefficients from the group $A$.
For a given quandle $2$-cocycle $\mu$ denote by the symbol $E(Q, A, \mu)$ the quandle $Q\times_{\psi\mu} A$. By formula \eqref{non-abelian-structure} the operation in $E(Q, A, \mu)$ has the form
$$(x, a) * (y, b) = \left(x * y, a + \mu(x, y)\right)$$
for $x, y \in Q$ and $a, b \in A$. The quandle $E(Q, A, \mu)$ is called {\it an abelian extension of $Q$ by $A$}. Plenty of well-known quandles are abelian extensions of smaller quandles (see \cite{Carter2} for details). All results which are faithfull for non-abelian extensions of quandles are also correct for abelian extensions of quandles. In particular, Theorem \ref{exterrr} emplies the following result
\begin{corollary}
Let $Q$ be a quandle, $A$ be an abelian group and $\mu : Q \times Q \to A$ be a quandle $2$-cocycle. Then there is an embedding $ \big({\rm Aut}(Q) \times {\rm Aut}(A) \big)_{\mu} \hookrightarrow  \Aut\big(E(Q, A, \mu)\big)$.
\end{corollary}
\bigskip

\section{Quasi-inner automorphisms}\label{sec8}
Recall that an automorphism $\varphi$ of a group $G$ is called {\it quasi-inner} if for any element $g \in G$ there exists an element $h \in G$ such that $\varphi(g) = h^{-1} g h$. Since the conjugation by the element $h$ defines the inner automorphism $\widehat{h}$ of the group $G$, then $\varphi$ is quasi-inner automorphism if for every element $g \in G$ there exists an inner automorphism $\widehat{h}$ of $G$ such that $\varphi(g) = \widehat{h}(g)$. Every inner automorphism of $G$ is obviously quasi-inner.

In 1913, Burnside formulated the following question: Is it true that any quasi-inner automorphism of a group is inner? Burnside \cite{Burnside} answered this question negatively by constructing an example of a finite group which has a quasi-inner automorphism which is not inner. For quandles we can formulate two different definitions of a quasi-inner automorphism:
\begin{enumerate}
\item An automorphism $\varphi$ of a quandle $Q$ is called {\it a quasi-inner automorphism in a strong sense} if for every element $x \in Q$ there exists an element $y\in Q$ such that $\varphi(x)=x*y$.
\item An automorphism $\varphi$ of a quandle $Q$ is called {\it a quasi-inner automorphism in a weak sense} if for every element $x \in Q$ there exists an inner automorphism $S \in \Inn(Q)$ such that $\varphi(x)=S(x)$.
\end{enumerate}

Every quasi-inner automorphism in a strong sense is obviously quasi-inner in a weak sense. However, in the general case the opposite is not correct since  the group of inner automorphisms $\Inn(Q)=\langle S_x~|~x\in Q \rangle$ does not have to coincide with the set $\{ S_x~|~x \in Q \}$. If $Q$ is a conjugation quandle ${\rm Conj}(G)$ for some group $G$, then the two definitions of a quasi-inner automorphism are the same. If $G$ is a $2$-step nilpotent group, then the two definitions of a quasi-inner automorphism are equivalent also for quandles $\Conj_n(G)$ for any $n$.

Denote by $\QInn (Q)$ the set of all quasi-inner in a weak sense automorphisms of a quandle $Q$. The set $\QInn (Q)$ is obviously a subgroup of $\Aut(Q)$ which contains the group of inner automorphisms $\Inn(Q)$. If $Q$ is the trivial quandle, then  both groups $\Inn(Q)$, $\QInn(Q)$ are trivial. The following question is an analogue of Burnside's problem for quandles.
\begin{question}\label{Quasi}
Does there exist a quandle $Q$ with $\Inn(Q) \neq \QInn(Q)$? 
\end{question}
If $G$ is a group with non-inner quasi-inner automorphism $\varphi$ constructed by Burnside, then the quandle $Q={\rm Conj}(G)$ has non-inner quasi-inner in a strong sense automorphism induced by the automorphism $\varphi$, what gives a negative answer to Question \ref{Quasi}. We are going to construct a simpler example of the quandle which is not a conjugation quandle and which possesses a non-inner quasi-inner automorphism.
\begin{lemma}\label{qaut}If $Q$ is a connected quandle, then ${\rm Aut}(Q)=\QInn(Q)$.
\end{lemma} 
\begin{proof}Let $\phi$ be an automorphism of $Q$. Since $Q$ is connected, for an element $\phi(x) \in Q$ there exists an inner automorphism $S \in \Inn(Q)$ such that $\phi(x)=S(x)$. Therefore ${\rm Aut}(Q)=\QInn(Q)$.
\end{proof}
\begin{proposition}\label{dihburn}
If $n\geq5$ is odd, then $\Inn(\R_n) \neq \QInn(\R_n)$.
\end{proposition}
\begin{proof}
By Theorem \ref{thm-dbs} for odd $n$ the group $\Aut( \R_n)$ is isomorphic to the group $\mathbb{Z}_n\rtimes \mathbb{Z}_n^*$, where $\mathbb{Z}_n^*$ is the multiplicative group of the ring $\mathbb{Z}_n$. Also by Theorem \ref{thm-dbs} for odd $n$ the group  $\Inn(\R_n)$ is isomorphic to the group  $\mathbb{Z}_n \rtimes \mathbb{Z}_2$. Therefore for odd $n\ge 5$ the groups $\Inn(\R_n)$ and $\Aut( \R_n)$ do not coincide. Since for odd $n$ the dihedral quandle $\R_n$ is connected, by Lemma \ref{qaut} we have $\QInn(\R_n)={\rm Aut}(\R_n)\neq \Inn(Q)$, i.~e.  there exists a quasi-inner automorphism of $\R_n$ which is not inner.
\end{proof}
\begin{remark}By Proposition \ref{r4-auto} the group $\Aut(\R_4)$ is isomorphic to  $\Inn (\R_4) \rtimes \langle \phi \rangle$, where $\phi$ is a non-inner automorphism of $\R_4$. Direct calculations show that $\phi$ is not a quasi-inner automorphism, therefore $\Inn (\R_4)=\QInn(\R_4)$. It would be interesting to explore whether a similar result holds for $\R_{2n}$ for $n>2$.
\end{remark}\bigskip
\section{Constructions of new quandles using automorphisms}\label{sec9}
In this section we describe some general approaches of constructing quandles using automorphism groups. If $G$ is a group, then a map $\phi:G \to \Aut(G)$ is said to be \textit{compatible} if for every element $x \in G$ the following diagram commutes
$$
\xymatrix{
G \ar[d]_{\phi(x)} \ar[r]^{\phi \hspace{3mm}} & \Aut(G) \ar[d]^{\iota_{\phi(x)}}\\
G \ar[r]^{\phi \hspace{3mm}} & \Aut(G),}
$$
where $\iota_{\phi(x)}$ denotes the inner automorphism of $\Aut(G)$ of the form $\iota_{\phi(x)}:\varphi\mapsto \phi(x)\varphi\phi(x)^{-1}$ for all $\varphi$ from ${\rm Aut}(G)$. For example, the map $\phi:x\mapsto \id_G$ and the map $\phi:x\mapsto\hat{x}^{-1}$ which maps every element $x$ to the inner automorphism $\hat{x}^{-1}: y\mapsto xyx^{-1}$ induced by $x^{-1}$ are both compatible. 

\begin{proposition}\label{new1}
Let $G$ be a group and $\phi: G \to \Aut(G)$ be a compatible map such that $\phi(x)(x)=x$ for all $x \in G$. Then the set $G$ with the operation  $x*y=\phi(y)(x)$ is a quandle. Moreover, if $\phi$ is an injective map which satisfies $\phi(xy)=\phi(y)\phi(x)$ for all $x,y\in G$, then  $x*y=y x  y^{-1}$.
\end{proposition}

\begin{proof} We need to check three axioms of a quandle. 
The equality $x*x=\phi(x)(x)=x$ for all $x \in G$ is given as the condition of the proposition, therefore the first quandle axiom is faithful. For an element $x$ the map $S_x:y\mapsto y*x=\phi(x)(y)$ is an automorphism of $G$ and therefore is a bijection, i.~e. the second quandle axiom is faithful. 
Since $\phi$ is a compatible map, for elements $y,z\in G$ we have $\phi(z)\phi(y)=\phi\left(\phi(z)(y)\right)\phi(z)$ and therefore
\begin{multline*}
(x*y)*z=\phi(z)(x*y)=\phi(z)\phi(y)(x)=\\=\phi(\phi(z)(y))\phi(z)(x)
=\phi(z)(x)*\phi(z)(y)=(x*z)*(y*z),
\end{multline*}
i.~e. the third quandle axiom is faithful and the set $G$ with the operation $x*y=\varphi(y)(x)$ is a quandle. 

If $\phi$ is a compatible map which satisfies the equality $\phi(xy)=\phi(y)\phi(x)$ for all $x,y\in G$, then  $\phi(x)\phi(y)=\phi\left({\phi(x)(y)}\right)\phi(x)$, therefore $\phi\left({\phi(x)(y)}\right)=\phi(xyx^{-1})$ and if $\phi$ is injective, then $y*x={\phi(x)(y)}=xyx^{-1}$.
\end{proof}
For the compatible map $\phi:G\to {\rm Aut}(G)$ of the form $x\mapsto \id_G$ the quandle constructed in Proposition \ref{new1} is trivial. While for the compatible map $\phi:G\to {\rm Aut}(G)$ which maps an element $x$ to the inner automorphism $\hat{x}^{-1}$ the quandle constructed in Proposition \ref{new1} is a quandle ${\rm Conj}_{-1}(G)$.
The following proposition gives another interesting construction of a quandle.
\begin{proposition}\label{new2}
Let $(Q_1, *)$ and $(Q_2,  \circ)$ be two quandles and let $\sigma: Q_1 \to  {\rm Conj}_{-1} \left(\Aut(Q_2) \right)$ and $\tau: Q_2 \to  {\rm Conj}_{-1} \left(\Aut(Q_1) \right)$ be two homomorphisms. Then the set $Q=Q_1 \sqcup Q_2$ with the operation
$$
x\star y=\begin{cases}
x*y,& x, y \in Q_1, \\
x\circ y,  &x, y \in Q_2, \\
{\tau(y)}(x),  &x \in Q_1, y \in Q_2, \\
{\sigma(y)}(x) &x \in Q_2, y \in Q_1.
\end{cases} 
$$
is a quandle if and only if the following conditions hold:
\begin{enumerate}
\item $\tau(z)(x)* y=\tau\left(\sigma(y)(z)\right)(x* y)$ for $x, y \in Q_1$ and $z \in Q_2$,
\item $\sigma(z)(x)\circ y=\sigma\left(\tau(y)(z)\right)(x\circ y)$ for $x, y \in Q_2$ and $z \in Q_1$.
\end{enumerate}
\end{proposition}
\begin{proof}We need to check three quandle axioms. For  $x\in Q$ the element $x\star x$ is equal to either $x*x$ or $x\circ x$ and is equal to $x$, therefore the first quandle axiom is faithful. For an element $x\in Q_1$ the map $S_x:y\mapsto y\star x$ acts as the map $y\mapsto y*x$ on $Q_1$ and as the map $\sigma(x):y\mapsto \sigma(x)(y)$ on $Q_2$. Therefore for $x\in Q_1$ the map $S_x$ is a bijection on $Q$. Analogically, for $x\in Q_2$ the map $S_x$ is a bijection on $Q$ and the second quandle axiom is faithful. Let $x,y,z$ be elements from $Q$. If all elements $x,y,z$ belong to $Q_i$, then the third quandle axiom obviously works for elements $x,y,z$ since it works in $Q_i$. So, let $x, y \in Q_1$ and $z \in Q_2$. In this case using equality (1) we have the equalities
\begin{align}
\notag(x\star y)\star z&=\tau(z)(x* y)=\tau(z)(x)* \tau(z)(y)=(x\star z)\star (y\star z),\\
\notag(z\star x)\star y &=  \sigma(x)(z)\star y=  \sigma(y) \sigma(x)(z)= \sigma(y) \sigma(x)\sigma(y)^{-1}(\sigma(y)(z))\\
\notag&=  \sigma(x*y) \big(\sigma(y)(z)\big)= \sigma(y)(z)\star(x\star y)= (z\star y)\star(x\star y),\\
\notag(x\star z)\star y&=\tau(z)(x)* y=\tau\big(\sigma(y)(z)\big)(x* y)=(x* y)\star \sigma(y)(z)=(x\star y)\star(z\star y),
\end{align}
i.~e. the third quandle axiom is faithfull for $x, y \in Q_1$ and $z \in Q_2$. Similarly, using equality (2) we obtain the third quandle axiom for $x, y \in Q_2$ and $z \in Q_1$. 
\end{proof}
For the trivial maps $\sigma:Q_1\to\{\id_{Q_2}\}$ and $\tau:Q_2\to\{\id_{Q_1}\}$ the required conditions (1) and (2) of Proposition \ref{new2} are faithful. So, for quandles $Q_1$ and $Q_2$ we can always define a quandle structure on $Q_1\sqcup Q_2$. 

As an another example, let $Q$ be \textit{an involutary quandle}, i.~e. a quandle where the equality $(x*y)*y=x$ holds for all $x,y\in Q$, and let both quandles $Q_1$, $Q_2$ be isomorphic to $Q$ with isomorphisms $\varphi:Q\to Q_1$ and $\psi:Q\to Q_2$. Using direct calculations it is easy to check that the maps $\sigma:\varphi(x)\mapsto S_{\psi(x)}$ and $\tau:\psi(x)\mapsto S_{\varphi(x)}$ satisfy conditions (1) and (2) of Proposition \ref{new2} and therefore we can define a quandle structure on $Q\sqcup Q$ for every involutary quandle $Q$ (in particular, for every core quandle $\Core(G)$ of a group $G$).

The quandle $Q=\langle x,y,z~|~x*y=z, x*z=x, z*x=z, z*y=x, y*x=y, y*z=y\rangle$ which we mentioned in Section \ref{sec3} as a quandle for which the map $Q\to G_Q$ is not injective, is obtained from two trivial quandles $Q_1=\{x,z\}$, $Q_2=\{y\}$ using the procedure from Proposition \ref{new2}, where the homomorphisms $\sigma$ maps $x$ and $z$ to $\id_{Q_2}$, and the homomorphism $\tau$ maps $y$ to the automorphism of $Q_1$ which permute $x$ and $z$.

\end{document}